\documentclass[12pt]{article}
\usepackage{amsmath, amsthm}
\usepackage{amssymb}
\newtheoremstyle{theorem}
  {10pt}		  
  {10pt}  
  {\sl}  
  {\parindent}     
  {\bf}  
  {. }    
  { }    
  {}     
\theoremstyle{theorem}
\newtheorem{theorem}{Theorem}

\newtheorem{pro}[theorem]{Proposition}

\newtheoremstyle{defi}
  {10pt}		  
  {10pt}  
  {\rm}  
  {\parindent}     
  {\bf}  
  {. }    
  { }    
  {}     
\theoremstyle{defi}
\newtheorem{definition}[theorem]{Definition}
\newtheorem{example}[theorem]{Example}
\newtheorem{remark}[theorem]{Remark}

\begin{document}

\title{Derivations and Centroids of Four Dimensional Associative Algebras}

\author{A.O. Abdulkareem$^1$, M.A. Fiidow$^2$ and I.S. Rakhimov$^{2,3}$\\
$^1$Department of Mathematics,
          College of Physical Sciences, \\ 
          Federal University of Agriculture Abeokuta\\
          PMB 2240, Alabata road, Abeokuta,\\
           Ogun State, Nigeria.\\
afeezokareem@gmail.com\\[2pt]
$^2$Department of Mathematics, Faculty of Science,\\ Universiti Putra Malaysia, UPM 43400 Serdang,\\ Selangor Darul Ehsan, Malaysia.\\
wilow8@gmail.com\\[2pt]
$^3$Institute for Mathematical Research (INSPEM),\\ Universiti Putra Malaysia, UPM 43400
 Serdang,\\ Selangor Darul Ehsan, Malaysia.\\
risamiddin@gmail.com}

\maketitle

\begin{abstract}
In this paper, we focus on derivations and centroids of four dimensional associative algebras. Using an existing classification result of low dimensional associative algebras, we describe the derivations and centroids of four dimensional associative algebras. We also identify algebra(s) that belong to the characteristically nilpotent class among the algebra of four dimensional associative algebras. 

{\bf AMS Subject Classification:} 16D70

{\bf Key Words and Phrases:}Derivation, Centroid, Associative algebras, Characteristically nilpotent.
\end{abstract}

\section{Introduction}
The interest in the study of derivations of algebras goes back to a paper by Jacobson \cite{J}. There, Jacobson proved that any Lie algebra over a field of characteristic zero which has non degenerate derivations is nilpotent. In the same paper, he asked for the converse. Dixmier and Lister \cite{DL} have given a negative answer to the converse of Jacobson's hypothesis by constructing an example of a nilpotent Lie algebra all of whose derivations are nilpotent (hence degenerate). Lie algebras whose derivations are nilpotent have been called characteristically nilpotent. The result of Dixmier and Lister in \cite{DL} is assumed to be the origin of the theory of characteristically nilpotent Lie algebras. A few years later in 1959, Leger and Togo \cite{RH3} published a paper showing the importance of the characteristically nilpotent Lie algebras. The results of Leger and Togo have been extended for the class of non associative algebras.

The theory of characteristically nilpotent Lie algebras constitutes an independent research object since 1955. Until then, most studies about Lie algebras were oriented to the classical aspects of the theory, such as semi-simple and reductive Lie algebras (see \cite{T}).

The structural theory of finite dimensional associative algebras have been treated  by \cite{peirc}. Many interesting results related to the problem have appeared since then. Further works  in this field can be found in \cite{hazlett}, \cite{mazzola1979}, \cite{mazzola1980}, \cite{poonen2008} and \cite{RRB}. Further development of the theory of associative algebras was in $80$-s of the last century when many open problems, remaining on unsolved since $30$-s, have been solved

Most classification problems of finite dimensional associative algebras have been studied for certain property(s) of associative algebras while the complete classification of associative algebras in general
is still an open problem.

Centroids of algebras play important role in the classification problems and in different areas of structure theory of algebras. The centroid of a Lie algebra is known to be a field and this fact plays an important role in the classification problem of finite dimensional extended affine Lie algebras over arbitrary field of characteristic zero (see \cite{BN}). Benkarta and Neher studied extended affine and root graded Lie algebras in \cite{BN}. Melville in \cite{Mel} studied the centroids of nilpotent Lie algebra. 
Centroids and derivations of associative algebras in dimension less than $4$ has been treated in \cite{FRS} and provided an impetus to further the study in higher dimension.

In this study, we concentrate on the derivations and centroids of associative algebras in dimension four. Using the classification result of associative algebras, we give a description of derivations and centroids of associative algebras in dimension $4$. 

%
%
%

\section{Preliminaries}
\label{}
\begin{definition}
An algebra $A$ over a field $\mathbb{K}$ is a vector space over $\mathbb{K}$ equipped with a bilinear map
\begin{equation*}
\lambda : A \times A \rightarrow A.
\end{equation*}
\end{definition}
\begin{definition}
An associative algebra $A$ is a vector space over a field $\mathbb{K}$ equipped with bilinear map $\Psi:A\times A :\longrightarrow A$ satisfying the associative law:
\begin{equation*}
\Psi(\Psi(x,y),z)= \Psi(x,\Psi(y,z)) \ for \ all \ x,y,z \in A.
\end{equation*}
\end{definition}
\begin{definition}
A Lie algebra $ L $ over a field $\mathbb{K}$ is an algebra satisfying the following conditions:
\begin{align*}
&[x,x]= 0,  \\
[[x,y],z]+ &[[y,z],x]+ [[z,x],y]= 0, \, \forall \ x,\ y, \ z \ \in L.
\end{align*}
\end{definition}
\begin{definition}
Let $(A_1,\cdot)$ and $(A_2,\circ)$ be two associative algebras over a field $\mathbb{K}$. A homomorphism between $A_1$ and $A_2$ is a $\mathbb{K}$-linear mapping $f:A_1\longrightarrow A_2$ such that $$f(x\cdot y)=f(x)\circ f(y) \ for \ all \ x,\ y \ \in A_1.$$
\end{definition}
The set of all homomorphism from $A_1$ to $A_2$ is denoted by $Hom_{\mathbb{K}}(A_1,A_2)$. If $A_1=A_2=A$, then it is an associative algebra with respect to composition operation and denoted by $Hom_{\mathbb{K}}(A)$. The linear mapping associated with $Hom_{\mathbb{K}}(A)$ is called an endomorphism. A bijective homomorphism is called isomorphism and the corresponding algebras are said to be isomorphic.
\begin{example}\label{AS}
Let $V$ be an $n$-dimensional vector space over a field $\mathbb{K}$. The set of all endomorphisms, $End(V)$, forms a vector space. The multiplication of two elements $f,g \in End(V)$ is defined by
\begin{equation*}
(f\circ g)(v)=f(g(v)), \ for \ all \ v \in V.
\end{equation*}
\end{example}
This product turns $End(V)$ into an associative algebra.
\begin{example}\label{EX1} Let the product of two elements in Example \ref{AS} above be defined by $[f,g]=f\circ g-g\circ f$. Then, $(End(V),[\cdot,\cdot])$ is a Lie algebra.
\end{example}

\begin{definition}
A linear transformation $d$ of an associative algebra $A$ is called a derivation if for any $x, y \in L$
$$ d(xy)=d(x)\cdot y+x\cdot d(y).$$
\end{definition}
The set of all derivations of an associative algebra A is denoted by $Der(A)$.
The following fact can easily be proven.
\begin{pro}
Let $A$ be an algebra. Then $Der(A)$ is a Lie algebra with respect to the bracket \ $[f,g]=f\circ g-g\circ f.$
\end{pro}
For  $a \in A,$ let $L_{a}$ and $R_{a}$ be the elements of $Hom(A)$ defined by:
$$L_a(x)=a\cdot x,\quad R_a(x)=x\cdot a,\quad x\in A.$$
\begin{theorem}\label{p2}
Let $ D\in Hom_\mathbb{K}(A) $. Then the following conditions are equivalent :\\
$1$. $d\in Der(A),$\\
$2$. $[d,L_x]=L_{d(x)}, \quad x\in A ,$\\
$3$. $[d,R_y]=R_{d(y)},\quad  y\in A\,.$\\
\end{theorem}
\begin{proof}
The proof of equivalence of $1$-$3$ is straightforward using the definition of derivation.
\end{proof}

\begin{definition}
Let $A$ be an arbitrary associative algebra
over a field $\mathbb{K}$: The centroid of $A$, $\Gamma(A)$ is defined by
$$\Gamma(A)=\{\phi\in End(A):\phi(xy)=\phi(x)y=x\phi(y)\}$$
\end{definition}
\begin{definition}
Let $H$ be a non empty subset of $A$. The set
$$Z_{A}(H)=\{x\in A:x\cdot H=H\cdot x=0\}.$$ is said to be centralizer of $H$ in $A$.  
\end{definition}
\noindent It must however be noted that $Z_{A}(A)=Z(A)$, the center of $A$. 
\begin{definition}
Let $\phi\in End(A)$. If $\phi(A)\subseteq Z(A)$ and
$\phi(A^{2})=0$, then $\phi$ is called a central derivation.
\end{definition}
\noindent The set of all central derivations is denoted by $C(A)$.

In the foregoing, we give a few earlier results on some properties of centroids of associative algebras which show the relationship between derivation and centroid. The proof of some of the facts given below can be found in \cite{FRS}. 
\begin{pro}[\cite{FRS}]
Let $A$ be an associative algebra over a field $\mathbb{K}$. If $\phi\in \Gamma(A)$ and $d\in Der(A)$ then $\phi\circ d$ is a derivation of $A$.
\end{pro}
\begin{pro}[\cite{FRS}]
Let $A$ be an associative algebra over a field $\mathbb{K}$. Then $$C(A)=\Gamma(A) \cap Der(A).$$
\end{pro}
\begin{pro}[\cite{FRS}]
Let $A$ be an associative algebra over a field $\mathbb{K}$. Then for any $d\in Der(A)$ and $\phi\in \Gamma(A)$:
\begin{enumerate}
\item The composition $d \circ \phi$ is in $\Gamma(A)$ if and only if $\phi \circ d$ is a central derivation of $A$;
\item The composition $d \circ \phi$ is a derivation of $A$ if and only if $[d\circ\phi]$ is a central derivation of $A$.
\end{enumerate}
\end{pro}
\section{Procedure for finding  derivations}
Let $A$  be an $n$-dimensional associative algebra and $d$ be its derivation.  Fix a basis $\{e_1, e_2,..., e_n\}$ of $A$. Particularly,
$$d\circ L_{e_{i}}(y)=L_{d(e_{i})}(y)+L_{e_i}\circ d(y)\quad \text{for basis vectors $e_{i}$}\quad i=1,2,\ldots,n.$$
we have\\
$\Longleftrightarrow$
$$(d\circ L_{e_i}-L_{e_{i}}\circ d)(y)=L_{d(e_{i})}(y)$$
An element $d$ of the derivations  being a linear transformation  of the vector space $A$ is represented in a matrix form $(d_{ij})_{i,j=1, 2,\cdots,n}$ i.e., $d(e_{i})=\sum\limits_{j=1}^{n}d_{ji}e_{j},\quad i=1, 2, \cdots, n$
\begin{equation*}
(d\circ L_{e_i}-L_{e_i}\circ d)(y)= \sum\limits_{j=1}^{n}d_{ji}L_{e_{j}}(y)
\end{equation*}
\begin{equation}\label{wacal}
d\circ L_{e_{i}}-L_{e_i}\circ d= \sum\limits_{j=1}^{n}d_{ji} L_{e_{j}}  \quad\quad    \forall  i=1,2,\cdots,n
\end{equation}
The last equations along with structure of $A$ give constraints for elements of the matrix $d$. Solving the system of equations, we can find the matrix
$$d=(d_{ij})=\left(%
\begin{array}{ccccccccccccccc}\label{SYQ}
  d_{11} & d_{12} & \cdots & d_{1n} \\
  \vdots  & \vdots  & \vdots & \vdots   \\
  d_{n1} & d_{n2} & \cdots & d_{nn} \\
\end{array}%
\right).$$
As mention earlier, we provide the classification results of $4$-dimensional associative algebra from \cite{RRB} which we use in our study. Note that $As^{m}_{n}$ denotes $m^{th}$ isomorphism class of associative algebra in dimension $n$.
\subsection{Four-dimensional  Associative algebras}
\begin{theorem}\label{4dim}
Any four-dimensional complex associative algebra can be included in one of the following isomorphism classes of algebras:
\begin{align*}
& As_{4}^{1} :e_{1}e_{1}=e_{3}, \quad e_{2}e_{2}=e_{4};\\
& As_{4}^{2} :e_{1}e_{2}=e_{3}, \quad e_{2}e_{1}=e_{4};\\
& As_{4}^{3} :e_{1}e_{2}=e_{4}, \quad e_{3}e_{1}=e_{4};\\
& As_{4}^{4} :e_{1}e_{1}=e_{4}, \quad e_{2}e_{2}=e_{2}, \quad  e_{2}e_{3}=e_{3}\\
& As_{4}^{5} :e_1e_1=e_4, \quad e_2e_2=e_2, \quad e_3e_2=e_3;\\
\end{align*}
\begin{align*}
& As_{4}^{6} :e_1e_2=e_3, \quad e_2e_1=e_4, \quad e_2e_2=-e_3;\\
& As_{4}^{7} :e_1e_2=e_3, \quad e_2e_1=-e_3, \quad e_2e_2=e_4;\\
& As_{4}^{8} :e_1e_2=e_4, \quad e_2e_1=-e_4, \quad e_3e_3=e_4;\\
& As_{4}^{9}(\alpha ) :e_1e_2=e_4, \quad  e_2e_1=\frac{1+\alpha}{1-\alpha}e_4, \quad e_2e_2=e_3;\\
& As_{4}^{10} :e_1e_1=e_1, \quad e_1e_2=e_2, \quad e_1e_3=e_3, \quad e_1e_4=e_4;\\
& As_{4}^{11} :e_1e_1=e_1, \quad e_1e_2=e_2, \quad e_1e_3=e_3, \quad e_4e_1=e_4;\\
& As_{4}^{12} :e_1e_1=e_1, \quad e_1e_2=e_2, \quad e_1e_4=e_4, \quad e_2e_3=e_4;\\
& As_{4}^{13} :e_1e_1=e_1, \quad e_1e_2=e_2, \quad e_3e_3=e_3, \quad e_3e_4=e_4;\\
& As_{4}^{14} :e_1e_1=e_1, \quad e_1e_4=e_4, \quad e_2e_1=e_2, \quad e_2e_4=e_3;\\
& As_{4}^{15} :e_1e_1=e_1, \quad e_1e_4=e_4, \quad e_2e_1=e_2, \quad e_3e_1=e_3 ;\\
& As_{4}^{16} :e_1e_1=e_1, \quad e_2e_1=e_2, \quad e_3e_1=e_3, \quad e_4e_1=e_4;\\
& As_{4}^{17} :e_1e_1=e_1, \quad e_2e_1=e_2, \quad e_3e_2=e_4, \quad e_4e_1=e_4;\\
& As_{4}^{18} :e_1e_1=e_1, \quad e_2e_1=e_2, \quad e_3e_3=e_3, \quad e_4e_3=e_4;\\
& As_{4}^{19} :e_1e_1=e_1, \quad e_2e_2=e_2, \quad e_2e_4=e_4, \quad e_3e_1=e_3;\\
& As_{4}^{20} :e_1e_1=e_1, \quad e_2e_2=e_2, \quad e_3e_3=e_3, \quad e_4e_4=e_4;\\
& As_{4}^{21} :e_1e_1=e_3, \quad e_1e_3=e_4, \quad e_2e_2=-e_4, \quad e_3e_1=e_4;\\
& As_{4}^{22} :e_1e_1=e_4, \quad e_1e_2=e_3, \quad e_2e_1=-e_3, \quad e_2e_2=-2e_3+e_4;\\
& As_{4}^{23}(\mu) :e_1e_1=e_4, \quad e_1e_2=e_3, \quad e_2e_1=-\mu e_4, \quad e_2e_2=-e_3;\\
& As_{4}^{24} :e_1e_1=e_4, \quad e_1e_2=e_4, \quad e_2e_1=-e_4, \quad e_3e_3=e_4;\\
& As_{4}^{25} :e_1e_1=e_4, \quad e_1e_4=-e_3, \quad e_2e_1=e_3, \quad e_4e_1=-e_3;\\
& As_{4}^{26} :e_1e_1=e_1, \quad e_1e_2=e_2, \quad e_1e_3=e_3, \quad e_1e_4=e_4, \quad e_2e_1=e_2;\\
& As_{4}^{27} :e_1e_1=e_1, \quad e_1e_2=e_2, \quad e_1e_4=e_4, \quad e_3e_1=e_3, \quad e_4e_1=e_4;\\
& As_{4}^{28} :e_1e_1=e_1, \quad e_1e_2=e_2, \quad e_2e_1=e_2, \quad e_3e_1=e_3, \quad e_4e_1=e_4;\\
& As_{4}^{29} :e_1e_1=e_1, \quad e_1e_3=e_3, \quad e_1e_4=e_4, \quad e_2e_2=e_2, \quad e_3e_2=e_3;\\
& As_{4}^{30} :e_1e_1=e_1, \quad e_1e_3=e_3, \quad e_2e_2=e_2, \quad e_2e_4=e_4, \quad e_4e_1=e_4;\\
& As_{4}^{31} :e_1e_1=e_1, \quad e_2e_2=e_2, \quad e_2e_3=e_3, \quad e_3e_1=e_3, \quad e_4e_1=e_4;\\
& As_{4}^{32} :e_1e_1=e_1, \quad e_2e_2=e_2, \quad e_2e_3=e_3, \quad e_3e_1=e_3, \quad e_4e_2=e_4;\\
& As_{4}^{33} :e_1e_1=e_1, \quad e_2e_2=e_2, \quad e_3e_2=e_3, \quad e_4e_3=e_3, \quad e_4e_4=e_4;\\
& As_{4}^{34} :e_1e_1=e_1, \quad e_2e_2=e_2, \quad e_3e_3=e_3, \quad e_3e_4=e_4, \quad e_4e_3=e_4;\\
& As_{4}^{35} :e_1e_1=e_4, \quad e_1e_2=\lambda e_4, \quad e_2e_1=-\lambda e_4, \quad e_2e_2=e_4, \quad e_3e_3=e_4;\\
& As_{4}^{36} :e_1e_1=e_4, \quad e_1e_4=-e_3, \quad e_2e_1=e_3, \quad e_2e_2=e_3, \quad e_4e_1=-e_3;\\
\end{align*}
\begin{align*}
& As_{4}^{37} :e_1e_2=e_4, \quad e_1e_3=e_4, \quad e_2e_1=-e_4, \quad e_2e_2=e_4, \quad e_3e_1=e_4;\\
& As_{4}^{38} :e_1e_1=e_1, \quad e_1e_2=e_2, \quad e_1e_3=e_3, \quad e_1e_4=e_4, \quad e_2e_1=e_2, \quad e_3e_1=e_3;\\
& As_{4}^{39} :e_1e_1=e_1, \quad e_1e_2=e_2, \quad e_1e_3=e_3, \quad e_1e_4=e_4, \quad e_3e_1=e_3, \quad e_3e_2=e_4;\\
& As_{4}^{40} :e_1e_1=e_1, \quad e_1e_2=e_2, \quad e_1e_3=e_3, \quad e_2e_1=e_2, \quad e_3e_1=e_3, \quad e_4e_1=e_4;\\
& As_{4}^{41} :e_1e_1=e_1, \quad e_1e_2=e_2, \quad e_2e_1=e_2, \quad e_3e_3=e_3, \quad e_3e_4=e_4, \quad e_4e_3=e_4;\\
& As_{4}^{42} :e_1e_1=e_1, \quad e_1e_3=e_3, \quad e_1e_4=e_4, \quad e_2e_2=e_2, \quad e_3e_1=e_3, \quad e_4e_2=e_4;\\
& As_{4}^{43} :e_1e_1=e_1, \quad e_1e_3=e_3, \quad e_2e_1=e_2, \quad e_2e_3=e_4, \quad e_3e_1=e_3, \quad e_4e_1=e_4;\\
& As_{4}^{44} :e_1e_1=e_1, \quad e_1e_3=e_3, \quad e_2e_2=e_2, \quad e_2e_4=e_4, \quad e_3e_1=e_3, \quad e_4e_1=e_4;\\
& As_{4}^{45} :e_1e_1=e_1, \quad e_1e_4=e_4, \quad e_2e_2=e_2, \quad e_2e_3=e_3, \quad e_3e_1=e_3, \quad e_4e_2=e_4;\\
& As_{4}^{46} :e_1e_1=e_1, \quad e_2e_2=e_2, \quad e_2e_3=e_3, \quad e_2e_4=e_4, \quad e_3e_1=e_3, \quad e_4e_1=e_4;\\
& As_{4}^{47} :e_1e_1=e_1, \quad e_2e_2=e_2, \quad e_2e_3=e_3, \quad e_2e_4=e_4, \quad e_3e_2=e_3, \quad e_4e_2=e_4;\\
& As_{4}^{48} :e_1e_1=e_2, \quad e_1e_2=e_3, \quad e_1e_3=e_4, \quad e_2e_1=e_3, \quad e_2e_2=e_4, \quad e_3e_1=e_4;\\
& As_{4}^{49} :e_1e_4=e_2, \quad e_2e_3=e_2, \quad e_3e_1=e_1, \quad e_3e_2=e_2, \quad e_3e_3=e_3, \quad e_4e_3=e_4;\\
& As_{4}^{50} :e_1e_1=e_1, \quad e_1e_2=e_2, \quad e_1e_3=e_3, \quad e_1e_4=e_4, \quad e_2e_1=e_2, \quad e_2e_2=e_4,\,\, e_4e_1=e_4;\\
& As_{4}^{51} :e_1e_1=e_1, \quad e_1e_2=e_2, \quad e_1e_3=e_3, \quad e_1e_4=e_4, \quad e_2e_1=e_2, \quad e_3e_1=e_3,\,\, e_4e_1=e_4;\\
& As_{4}^{52} :e_1e_1=e_1, \quad e_1e_2=e_2, \quad e_1e_4=e_4, \quad e_2e_1=e_2, \quad e_2e_2=e_4, \quad e_3e_1=e_3,\,\, e_4e_1=e_4;\\
& As_{4}^{53} :e_1e_1=e_1, \quad e_2e_2=e_2, \quad e_2e_3=e_3, \quad e_2e_4=e_4, \quad e_3e_2=e_3, \quad e_3e_3=e_4,\,\, e_4e_2=e_4;\\
& As_{4}^{54} :e_1e_1=e_1, \quad e_1e_2=e_2, \quad e_2e_3=e_1, \quad e_2e_4=e_2, \quad e_3e_1=e_3, \quad e_3e_2=e_4,\,\, e_4e_3=e_3, \\& e_4e_4=e_4;\\
& As_{4}^{55} :e_1e_1=e_1, \quad e_1e_2=e_2, \quad e_1e_3=e_3, \quad e_1e_4=e_4, \quad e_2e_1=e_2,\,\, e_2e_2=e_4, \quad e_3e_1=e_3, \\& e_4e_1=e_4;\\
& As_{4}^{56} :e_1e_1=e_1, \quad e_1e_2=e_2, \quad e_1e_3=e_3, \quad e_1e_4=e_4, \quad e_2e_1=e_2,
e_2e_3=e_4, \quad e_3e_1=e_3, \\& e_3e_2=-e_4, \ \ e_4e_1=e_4;\\
& As_{4}^{57} :e_1e_1=e_1, \quad e_1e_2=e_2, \quad e_1e_3=e_3, \quad e_1e_4=e_4, \quad e_2e_1=e_2, e_2e_2=e_3, \quad e_2e_3=e_4, \\& e_3e_1=e_3, \quad e_3e_2=e_4, \quad e_4e_1=e_4;\\
& As_{4}^{58} :e_1e_1=e_1, \quad e_1e_2=e_2, \quad e_1e_3=e_3, \quad e_1e_4=e_4, \quad e_2e_1=e_2,
\,\,\,\,\, e_2e_2=-e_4, \\& e_2e_3=-e_4, \quad e_3e_1=e_3, \qquad e_3e_2=e_4, \quad e_4e_1=e_4;
\end{align*}
for all $\alpha \in \mathbb{C}\setminus \{1\} \mbox{ and } \lambda, \mu \in \mathbb{C}$.
\end{theorem}
In what follows, the following notations are introduced:
\begin{enumerate}
\item IC: isomorphism classes of algebras.
\item Dim: Dimensions of the algebra of derivations.
\item $m_{2}=d_{11}+d_{21}$,
\item $m_{3}=d_{33}-d_{11}$,
\item $m_{4}=\frac{2}{1-\alpha} d_{12}$,
\item  $m_{5}=d_{22}+d_{33}$,
\item $m_{6}=d_{33}-d_{22}$.
\end{enumerate}

Since the classification of four-dimensional associative algebra is already known from Theorem \ref{4dim}, Table \ref{dd} shows the derivations four-dimensional associative algebra as given below:
\begin{theorem}
The derivation of four-dimensional associative algebras has the following form:
\end{theorem}
\begin{table}[ht!]
\caption{Derivations of four-dimensional associative algebras}\label{dd}
\begin{center}
\small
\begin{tabular}{|c|c|c||c|c|c|} \hline
\textbf{IC} & \textbf{Derivation} & \textbf{Dim} & \textbf{IC} & \textbf{Derivation} & \textbf{Dim}\\
\hline
$As_{4}^{1}$ &$\left(%
\begin{array}{cccc}
  d_{11} & 0 & 0&0 \\
  0 & d_{22}& 0&0 \\
  d_{31}&d_{32}&2d_{11}&0 \\
  d_{41}&d_{42}&0&2d_{22}
\end{array}%
\right)$&$6$&

$As_{4}^{2}$ &$\left(%
\begin{array}{cccc}
   d_{11} & 0 & 0&0 \\
   0 & d_{22}& 0&0 \\
  d_{31}&d_{32}&m_{2}&0 \\
   d_{41}&d_{42}& 0&m_{2}
\end{array}%
\right)$&$6$\\ \hline

$As_{4}^{3}$&$\left(%
\begin{array}{cccc}
  d_{11} & 0 & 0&0 \\
   d_{21} & d_{22}& 0&0 \\
  -d_{21}&0&d_{22}&0 \\
  d_{41}&d_{42}&d_{43}&m_{2}
\end{array}%
\right)$&$6$&
$As_{4}^{4}$ &$\left(%
\begin{array}{cccc}
  d_{11} & 0 & 0&0 \\
   0 & 0& 0&0 \\
  0&d_{32}&d_{33}&0 \\
  d_{41}&0& 0&2d_{11}
\end{array}%
\right)$&$4$\\
\hline

$As_{4}^{5}$ &$\left(%
\begin{array}{cccc}
  d_{11} & 0 & 0&0 \\
   0 & 0& 0&0 \\
  0&d_{32}&d_{33}&0 \\
  d_{41}&0& 0&2d_{11}
\end{array}%
\right)$&$4$&

$As_{4}^{6}$& $\left(%
\begin{array}{cccc}
  d_{11} & 0 & 0&0 \\
   0 & d_{11}& 0&0 \\
  d_{31}&d_{32}&2d_{11}&0 \\
  d_{41}&d_{42}& 0&2d_{11}
\end{array}%
\right)$&$5$ \\ \hline

$As_{4}^{7}$&$\left(%
\begin{array}{cccc}
  d_{11} & d_{12} & 0&0 \\
   0 & d_{22}& 0&0 \\
  d_{31}&d_{32}&m_{2}&0 \\
  d_{41}&d_{42}& 0&2d_{22}
\end{array}%
\right)$&$7$&

$As_{4}^{8}$& $\left(%
\begin{array}{cccc}
  d_{11} & d_{12} & 0&0 \\
   d_{21} & m_{3}& 0&0 \\
  0&0&d_{33}&0 \\
  d_{41}&d_{42}& d_{43}&2d_{33}
\end{array}%
\right)$&$7$\\ \hline
\end{tabular}
\end{center}
\begin{flushright}
Continued to the next page
\end{flushright}
\end{table}
%
\clearpage
\begin{table}[ht!]
\begin{center}
\small
\begin{tabular}{|c|c|c||c|c|c|} \hline
\textbf{IC} & \textbf{Derivation} & \textbf{Dim} & \textbf{IC} & \textbf{Derivation} & \textbf{Dim}\\
\hline

$As_{4}^{9}(\alpha )$& $\left(%
\begin{array}{cccc}
  d_{11} & d_{12} & 0&0 \\
  0 & d_{22}& 0&0 \\
  d_{31}&d_{32}&2d_{22}&0 \\
  d_{41}&d_{42}&m_{4}&m_{2}
\end{array}%
\right)$&$7$&

$As_{4}^{10}$& $\left(%
\begin{array}{cccc}
  0 & 0 & 0&0 \\
  d_{21} & d_{22}& d_{23}&d_{24} \\
  d_{31}&d_{32}&d_{33}&d_{34} \\
  d_{41}&d_{42}&d_{43}&d_{44}
\end{array}%
\right)$&$12$\\
\hline
$As_{4}^{11}$& $\left(%
\begin{array}{cccc}
  0 & 0 & 0&0 \\
  d_{21} & d_{22}& d_{23}&0 \\
  d_{31}& d_{32}&d_{33}&0 \\
  d_{41}&0&0&d_{44}
\end{array}%
\right)$&$8$&

 $As_{4}^{12}$& $\left(%
\begin{array}{cccc}
  0 & 0 & 0&0 \\
  d_{21} & d_{22}& 0&0 \\
  0& 0&d_{33}&0 \\
  d_{41}&d_{42}&-d_{21}&m_{5}
\end{array}%
\right)$&$5$\\
\hline
$As_{4}^{13}$& $\left(%
\begin{array}{cccc}
  0 & 0 & 0&0 \\
  d_{21} & d_{22}& 0&0 \\
  0& 0&0&0 \\
  0&0&d_{43}&d_{44}
\end{array}%
\right)$&$4$&

$As_{4}^{14}$& $\left(%
\begin{array}{cccc}
  0 & 0 & 0&0 \\
  d_{21} & d_{22}& 0&0 \\
  0& d_{32}&d_{33}&d_{21} \\
  d_{32}&0&0&m_{6}
\end{array}%
\right)$&$4$\\
 \hline
$As_{4}^{15}$& $\left(%
\begin{array}{cccc}
  0 & 0 & 0&0 \\
  d_{21} & d_{22}& d_{23}&0 \\
  d_{31}& d_{32}&d_{33}&0 \\
  d_{41}&0&0&d_{44}
\end{array}%
\right)$&$8$&
$As_{4}^{16}$& $\left(%
\begin{array}{cccc}
  0 & 0 & 0&0 \\
  d_{21} & d_{22}& d_{23}&d_{24} \\
  d_{31}& d_{32}&d_{33}&d_{34} \\
  d_{41}&d_{42}&d_{43}&d_{44}
\end{array}%
\right)$&$12$\\ \hline
$As_{4}^{17}$& $\left(%
\begin{array}{cccc}
  0 & 0 & 0&0 \\
  d_{21} & d_{22}& 0&0 \\
  0& 0&d_{33}&0 \\
  d_{41}&d_{42}&-d_{21}&m_{5}
\end{array}%
\right)$&$5$&
$As_{4}^{18}$& $\left(%
\begin{array}{cccc}
  0 & 0 & 0&0 \\
  d_{21} & d_{22}& 0&0 \\
  0& 0&0&0 \\
  0&0&d_{43}&d_{44}
\end{array}%
\right)$&$4$\\
\hline
$As_{4}^{19}$& $\left(%
\begin{array}{cccc}
  0 & 0 & 0&0 \\
  0 & 0& 0&0 \\
 d_{31}& 0&d_{33}&0 \\
  0&d_{42}&0&d_{44}
\end{array}%
\right)$&$4$&
$As_{4}^{20}$& $\left(%
\begin{array}{cccc}
  0 & 0 & 0&0 \\
  0 & 0& 0&0 \\
 0& 0&0&0 \\
  0&0&0&0
\end{array}%
\right)$&$0$\\
\hline
$As_{4}^{21}$
& $\left(%
\begin{array}{cccc}
  d_{11}& 0 & 0&0 \\
  d_{21} & \frac{3}{2} d_{11}& 0&0 \\
 d_{31}& d_{21}&2d_{11}&0 \\
  d_{41}&d_{42}&2d_{31}&3d_{11}
\end{array}%
\right)$&$5$&
$As_{4}^{22}$& $\left(%
\begin{array}{cccc}
  d_{11}& 0 & 0&0 \\
  0 & 0& 0&0 \\
 d_{31}& d_{32}&2d_{11}&0 \\
  d_{41}&d_{42}&0&2d_{11}
\end{array}%
\right)$&$5$\\
\hline
$As_{4}^{23}(\mu)$& $\left(%
\begin{array}{cccc}
  d_{11}& 0 & 0&0 \\
  0 & d_{11}& 0&0 \\
 d_{31}& d_{32}&2d_{11}&0 \\
  d_{41}&d_{42}&0&2d_{11}
\end{array}%
\right)$&$5$&
$As_{4}^{24}$& $\left(%
\begin{array}{cccc}
  d_{11}& 0 & 0&0 \\
  d_{21} & d_{11}& 0&0 \\
 0& 0&d_{11}&0 \\
  d_{41}&d_{42}&d_{43}&2d_{11}
\end{array}%
\right)$&$5$\\
\hline
$As_{4}^{25}$ & $\left(%
\begin{array}{cccc}
  d_{11}& 0 & 0&0 \\
  d_{21} & 2d_{11}& 0&0 \\
 d_{31}& d_{32}&3d_{11}&m_{7} \\
  d_{41}&0&0&2d_{11}
\end{array}%
\right)$&$5$&
$As_{4}^{26}$& $\left(%
\begin{array}{cccc}
  0& 0& 0&0\\
  0 & d_{22}& 0&0 \\
 d_{31}& 0&d_{33}&d_{34}\\
  d_{41}&0&d_{43}&d_{44}
\end{array}%
\right)$&$7$\\
\hline
\end{tabular}
\end{center}
\begin{flushright}
Continued to the next page
\end{flushright}
\end{table}
%
\clearpage
\begin{table}[ht!]
\begin{center}
\small
\begin{tabular}{|c|c|c||c|c|c|} \hline
\textbf{IC} & \textbf{Derivation} & \textbf{Dim} & \textbf{IC} & \textbf{Derivation} & \textbf{Dim}\\
\hline

$As_{4}^{27}$& $\left(%
\begin{array}{cccc}
  0& 0& 0&0\\
  d_{21}& d_{22}& 0&0 \\
 d_{31}& 0&d_{33}&d_{34}\\
  0&0&0&d_{44}
\end{array}%
\right)$&$5$&
$As_{4}^{28}$& $\left(%
\begin{array}{cccc}
  0& 0& 0&0\\
  0& d_{22}& 0&0 \\
 d_{31}& 0&d_{33}&d_{34}\\
  d_{41}&0&d_{43}&d_{44}
\end{array}%
\right)$&$7$\\
\hline
$As_{4}^{29}$ & $\left(%
\begin{array}{cccc}
   0& 0& 0&0\\
  0& 0& 0&0 \\
 d_{31}& -d_{31}&d_{33}&0\\
  d_{41}&0&0&d_{44}
\end{array}%
\right)$&$4$&
$ As_{4}^{30}$& $\left(%
\begin{array}{cccc}
  0& 0& 0&0\\
  0& 0& 0&0 \\
 d_{31}& 0&d_{33}&0\\
  d_{41}&-d_{41}&0&d_{44}
\end{array}%
\right)$&$4$\\
\hline
$As_{4}^{31}$ & $\left(%
\begin{array}{cccc}
  0& 0& 0&0\\
  0& 0& 0&0 \\
 d_{31}& -d_{31}&d_{33}&0\\
  d_{41}&0&0&d_{44}
\end{array}%
\right)$&$4$&
$As_{4}^{32}$& $\left(%
\begin{array}{cccc}
  0& 0& 0&0\\
  0& 0& 0&0 \\
 d_{31}&-d_{31}&d_{33}&0\\
  0&d_{42}&0&d_{44}
\end{array}%
\right)$&$4$\\
\hline
$As_{4}^{33}$& $\left(%
\begin{array}{cccc}
  0& 0& 0&0\\
  0& 0& 0&0 \\
 0&d_{32}&d_{33}&-d_{32}\\
  0&0&0&0
\end{array}%
\right)$&$2$&
$As_{4}^{34}$& $\left(%
\begin{array}{cccc}
  0& 0& 0&0\\
  0& 0& 0&0 \\
 0&0&0&0\\
  0&0&0&d_{44}
\end{array}%
\right)$&$1$\\ \hline
$As_{4}^{35}$& $\left(%
\begin{array}{cccc}
  d_{11}& -d_{21}& 0&0\\
  d_{21}& d_{11}& 0&0 \\
 0&0&0&0\\
  d_{41}&d_{42}&d_{43}&2d_{11}
\end{array}%
\right)$&$5$&
$ As_{4}^{36}$& $\left(%
\begin{array}{cccc}
  0& 0& 0&0\\
  d_{21}& 0& 0&0 \\
 d_{31}&d_{32}&0&m_{7}\\
  d_{41}&d_{21}&0&0
\end{array}%
\right)$&$4$\\
\hline
$As_{4}^{37}$ & $\left(%
\begin{array}{cccc}
  d_{11}& 0& 0&0\\
  d_{21}& d_{11}& 0&0 \\
 0&-d_{21}&d_{11}&0\\
  d_{41}&d_{42}&d_{43}&2d_{11}
\end{array}%
\right)$&$5$&
$As_{4}^{38}$ & $\left(%
\begin{array}{cccc}
  0& 0& 0&0\\
  0& d_{22}& d_{23}&0 \\
 0&d_{32}&d_{33}&0\\
  d_{41}&0&0&d_{44}
\end{array}%
\right)$&$6$\\
\hline
$As_{4}^{39}$ & $\left(%
\begin{array}{cccc}
  0& 0& 0&0\\
  d_{21}& d_{22}&0&0 \\
 0&0&d_{33}&0\\
  d_{41}&d_{42}&d_{21}&m_{5}
\end{array}%
\right)$&$5$&
$As_{4}^{40}$ & $\left(%
\begin{array}{cccc}
  0& 0& 0&0\\
  0& d_{22}& d_{23}&0 \\
 0&d_{32}&d_{33}&0\\
  d_{41}&0&0&d_{44}
\end{array}%
\right)$&$6$\\
\hline
$As_{4}^{41}$ & $\left(%
\begin{array}{cccc}
  0& 0& 0&0\\
  0& d_{22}&0&0 \\
 0&0&0&0\\
  0&0&0&d_{44}
\end{array}%
\right)$&$2$&
$As_{4}^{42}$  & $\left(%
\begin{array}{cccc}
  0& 0& 0&0\\
  0& 0&0&0 \\
 0&0&d_{33}&0\\
  d_{41}&-d_{41}&0&d_{44}
\end{array}%
\right)$&$3$\\
\hline
$As_{4}^{43}$ & $\left(%
\begin{array}{cccc}
  0& 0& 0&0\\
  d_{21}&d_{22}&0&0 \\
 0&0&d_{33}&0\\
  d_{41}&d_{42}&d_{21}&m_{5}
\end{array}%
\right)$&$5$&
$As_{4}^{44}$& $\left(%
\begin{array}{cccc}
  0& 0& 0&0\\
  0&0&0&0 \\
 0&0&d_{33}&0\\
  d_{41}&-d_{41}&0&d_{44}
\end{array}%
\right)$&$3$\\
\hline
\end{tabular}
\end{center}
\begin{flushright}
Continued to the next page
\end{flushright}
\end{table}
\clearpage
%
%
%
\begin{table}[ht!]
\begin{center}
\small
\begin{tabular}{|l|c|c||l|c|c|} \hline
\textbf{IC} & \textbf{Derivation} & \textbf{Dim} & \textbf{IC} & \textbf{Derivation} & \textbf{Dim}\\
\hline
$As_{4}^{45}$ & $\left(%
\begin{array}{cccc}
  0& 0& 0&0\\
  0&0&0&0 \\
 d_{31}&-d_{31}&d_{33}&0\\
  d_{41}&-d_{41}&0&d_{44}
\end{array}%
\right)$&$4$&
$As_{4}^{46}$ & $\left(%
\begin{array}{cccc}
  0& 0& 0&0\\
  0&0&0&0 \\
 d_{31}&-d_{31}&d_{33}&d_{34}\\
  d_{41}&-d_{41}&d_{43}&d_{44}
\end{array}%
\right)$&$6$\\
\hline

$ As_{4}^{47}$ & $\left(%
\begin{array}{cccc}
  0& 0& 0&0\\
  0&0&0&0 \\
 0& 0&d_{33}&d_{34}\\
  0&0&d_{43}&d_{44}
\end{array}%
\right)$&$4$&
$As_{4}^{48}$ & $\left(%
\begin{array}{cccc}
  d_{11}& 0& 0&0\\
  d_{21}&2d_{11}&0&0 \\
 d_{31}& 2d_{21}&3d_{11}&0\\
  d_{41}&2d_{31}&3d_{21}&4d_{11}
\end{array}%
\right)$&$4$\\
\hline
$As_{4}^{49}$ & $\left(%
\begin{array}{cccc}
  d_{11}& 0& d_{13}&0\\
  d_{21}&d_{22}&0&-d_{13} \\
 0& 0&0&0\\
  0&0&-d_{21}&m_{1}
\end{array}%
\right)$&$4$&
$As_{4}^{50}$ & $\left(%
\begin{array}{cccc}
  0& 0& 0&0\\
  0&d_{22}&0&0 \\
 d_{31}& 0&d_{33}&0\\
  0&d_{42}&0&2d_{22}
\end{array}%
\right)$&$4$\\
\hline
$As_{4}^{51}$ & $\left(%
\begin{array}{cccc}
  0& 0& 0&0\\
  0&d_{22}&d_{23}&d_{24} \\
 0& d_{32}&d_{33}&d_{34}\\
  0&d_{42}&d_{43}&d_{44}
\end{array}%
\right)$&$9$&
$As_{4}^{52}$  & $\left(%
\begin{array}{cccc}
  0& 0& 0&0\\
  0&d_{22}&0&0\\
 d_{31}& 0&d_{33}&0\\
  0&d_{42}&0&2d_{22}
\end{array}%
\right)$&$4$\\
\hline
$As_{4}^{53}$  & $\left(%
\begin{array}{cccc}
  0& 0& 0&0\\
  0&0&0&0\\
 0& 0&d_{33}&0\\
  0&0&d_{43}&2d_{33}
\end{array}%
\right)$&$2$&
$As_{4}^{54}$   & $\left(%
\begin{array}{cccc}
  0& d_{12}& -d_{21}&0\\
  d_{21}&d_{22}&0&-d_{21}\\
- d_{12}& 0&-d_{22}&-d_{12}\\
  0&-d_{12}&d_{13}&0
\end{array}%
\right)$&$3$\\
\hline
$As_{4}^{55}$ & $\left(%
\begin{array}{cccc}
  0& 0& 0&0\\
  0&d_{22}&0&0\\
 0& d_{32}&d_{33}&0\\
  0&d_{42}&d_{43}&2d_{22}
\end{array}%
\right)$&$5$&
$As_{4}^{56}$ & $\left(%
\begin{array}{cccc}
  0& 0& 0&0\\
  0&d_{22}&d_{23}&0\\
 0& d_{32}&d_{33}&0\\
  0&d_{42}&d_{43}&m_{5}
\end{array}%
\right)$&$6$\\
\hline
$As_{4}^{57}$& $\left(%
\begin{array}{cccc}
  0& 0& 0&0\\
  0&d_{22}&0&0\\
 0& d_{32}&2d_{22}&0\\
  0&d_{42}&2d_{32}&3d_{22}
\end{array}%
\right)$&$3$&
$As_{4}^{58}$& $\left(%
\begin{array}{cccc}
  0& 0& 0&0\\
  0&d_{22}&0&0\\
 0& d_{32}&d_{22}&0\\
  0&d_{42}&d_{43}&2d_{22}
\end{array}%
\right)$&$4$\\
\hline
\end{tabular}
\end{center}
\end{table}
\begin{proof}
From Theorem \ref{4dim}, we provide the proof only for one case to illustrate the approach used, the other cases can be carried out similarly with little or no modification(s).\\
Let us consider $As_{4}^{1}$. Applying the system of equation (\ref{wacal}), we get $d_{12}=d_{13}=d_{14}=d_{21}=d_{23}=d_{24}=d_{34}=d_{43}=0$, $d_{33}=2d_{11}, d_{44}=2d_{22}$

\noindent Hence, the derivation for $As_{4}^{1}$ are given as follows $d=$ $\left(%
\begin{array}{cccc}
  d_{11} & 0 & 0&0 \\
  0 & d_{22}& 0&0 \\
  d_{31}&d_{32}&2d_{11}&0 \\
  d_{41}&d_{42}&0&2d_{22}
\end{array}%
\right).$

\noindent So, $d_1=\left( \begin {array}{cccc} 1&0&0&0 \\ 0&0&0&0 \\ 0&0&2&0 \\0&0&0&0 \end{array} \right)$,
 $d_2=\left( \begin {array}{cccc} 0&0&0&0 \\ 0&1&0&0 \\ 0&0&0&0 \\ 0&0&0&2 \end{array} \right)$,
 $d_3=\left( \begin {array}{cccc} 0&0&0&0 \\ 0&0&0&0 \\ 1&0&0&0 \\ 0&0&0&0 \end{array} \right)$, \\\\

\noindent $d_4=\left( \begin {array}{cccc} 0&0&0&0 \\ 0&0&0&0 \\ 0&1&0&0 \\ 0&0&0&0 \end{array} \right)$, $\quad d_5=\left( \begin {array}{cccc} 0&0&0&0 \\ 0&0&0&0 \\ 0&0&0&0 \\1&0&0&0 \end{array} \right)$,$\quad d_6=\left( \begin {array}{cccc} 0&0&0&0 \\ 0&0&0&0 \\ 0&0&0&0 \\0&1&0&0 \end{array} \right)$.

\noindent is a basis of $Der(A)$ and $dimDer(A)=6.$\\
The derivation of the remaining parts of dimension four algebras can be carried out in a similar manner as shown above.
\end{proof}
\begin{remark}
\noindent
\begin{enumerate}
\item There is  one class of  characteristically nilpotent associative algebras in the list of isomorphism classes of  four dimensional associative algebras.
\item The dimensions of the derivation algebras in this case  vary between zero and nine.
\end{enumerate}
\end{remark}
\section{Description of centroids of four dimensional associative algebras}
In this section, we give the description of centroids of four-dimensional complex associative algebras.
Let $\{e_1, e_2, e_3, e_4 \}$ be  a basis of a $4$-dimensional associative algebra $A$, then $e_ie_j=\sum\limits_{k=1}^{4}\gamma_{ij}^{k}e_{k}.$
The coefficients, $\{\gamma_{ij}^{k}\}\in \mathbb{C}^{4^{3}}$, of the above linear combinations are called the structure constants of $A$ on the basis $\{e_1, e_2, e_3, e_4 \}.$ An element $\phi$ of the centroid $\Gamma(A)$, being a linear transformation of the vector space $A$, is represented in a matrix form $(a_{ij})_{i,j=1,2,...,4},$ i.e. $\phi(e_i)=\sum\limits_{j=1}^{4}a_{ji}e_{j},$ $i=1,2,...,4.$
must satisfy the following systems of equations:
\begin{equation}\label{qaaq}
\sum_{k=1}^{4}(\gamma_{ij}^{t}a_{kt}-a_{ti}\gamma_{tj}^k)=0
\end{equation}
\begin{equation}\label{qaaaq}
\sum_{k=1}^{4}(\gamma_{ij}^{k}a_{kt}-a_{tj}\gamma_{it}^k)=0
\end{equation}
It is observed that if the structure constants $\{\gamma_{ij}^{k}\}$ of the associative algebra $A$ are given, then, in order to describe its centroid one has to solve the system of equations above with respect to the $a_{ij}$, $i,j=1,2,...,4.$ 
Again, we use  classification results of four-dimensional complex associative algebras from \cite{RRB} to compute the centroids of four dimensional complex associative algebras.
\begin{theorem}
The centroids of four dimensional complex associative algebras are given as follows:
\end{theorem}
\begin{table}[ht!]
\caption{Centroids of four-dimensional associative algebras}
\begin{center}
\small
\begin{tabular}{|c|c|c||c|c|c|c|c|} \hline
\textbf{IC}&\textbf{Centroid }&\textbf{Dim}&\textbf{IC}&\textbf{Centroid }&\textbf{Dim}\\
& \textbf{$\Gamma(A)$} & && $\Gamma(A)$&  \\
\hline
$As_{4}^{1}$&$\left(%
\begin{array}{cccc}
  a_{11} & 0 & 0&0 \\
  0 & a_{22}& 0&0 \\
  0&0&a_{11}&0 \\
  0&0&0&a_{11}
\end{array}%
\right)$&$2$&
$As_{4}^{2}$&$\left(%
\begin{array}{cccc}
   a_{11} & 0 & 0&0 \\
   0 & a_{11}& 0&0 \\
  a_{31}&a_{32}&a_{11}&0 \\
   a_{41}&a_{42}& 0&a_{11}
\end{array}%
\right)$&$5$\\
\hline
$As_{4}^{3}$ &$\left(%
\begin{array}{cccc}
  a_{11} & 0 & 0&0 \\
 0 & a_{11} & 0&0 \\
  0&0&a_{11} &0 \\
  a_{41}&a_{42}&a_{43}&a_{11}
\end{array}%
\right)$&$4$&
$As_{4}^{4}$&$\left(%
\begin{array}{cccc}
  a_{11} & 0 & 0&0 \\
   0 & a_{22}& 0&0 \\
  0&0&a_{22}&0 \\
  a_{41}&0& 0&a_{11}
\end{array}%
\right)$&$3$\\
\hline
$As_{4}^{5}$&$\left(%
\begin{array}{cccc}
  a_{11} & 0 & 0&0 \\
   0 & a_{22}& 0&0 \\
  0&0&a_{22}&0 \\
  a_{41}&0& 0&a_{11}
\end{array}%
\right)$&$3$&
$As_{4}^{6}$& $\left(%
\begin{array}{cccc}
  a_{11} & 0 & 0&0 \\
   0 &  a_{11}& 0&0 \\
  a_{31}&a_{32}& a_{11}&0 \\
  a_{41}&a_{42}& 0& a_{11}
\end{array}%
\right)$&$5$\\
\hline
\end{tabular}
\end{center}
\begin{flushright}
Continued to the next page
\end{flushright}
\end{table}
\clearpage
\begin{table}[ht!]
\begin{center}
\small
\begin{tabular}{|c|c|c||c|c|c|c|c|} \hline
\textbf{IC}&\textbf{Centroid }&\textbf{Dim}&\textbf{IC}&\textbf{Centroid }&\textbf{Dim}\\
& \textbf{$\Gamma(A)$} & && $\Gamma(A)$&  \\
\hline

$As_{4}^{7}$&$\left(%
\begin{array}{cccc}
  a_{11} & a_{12} & 0&0 \\
   0 & a_{11}& 0&0 \\
  a_{31}&a_{32}& a_{11}& a_{12}\\
  a_{41}&a_{42}& 0& a_{11}
\end{array}%
\right)$&$6$&
$As_{4}^{8}$& $\left(%
\begin{array}{cccc}
  a_{11} & 0 & 0&0 \\
   0 & a_{11}& 0&0 \\
  0&0&a_{11}&0 \\
  a_{41}&a_{42}& a_{43}&a_{11}
\end{array}%
\right)$&$4$\\
\hline
$As_{4}^{9}(\alpha )$& $\left(%
\begin{array}{cccc}
  a_{11} & 0 & 0&0 \\
  0 & a_{11}& 0&0 \\
  a_{31}&a_{32}&a_{11}&0 \\
  a_{41}&a_{42}&0&a_{11}
\end{array}%
\right)$&$5$&
$As_{4}^{10}$& $\left(%
\begin{array}{cccc}
 a_{11} & 0 & 0&0 \\
  a_{21} &  a_{11}& 0&0 \\
  a_{31}&0& a_{11}&0 \\
  a_{41}&0&0& a_{11}
\end{array}%
\right)$&$4$\\
\hline
$As_{4}^{11}$& $\left(%
\begin{array}{cccc}
  a_{11} & 0 & 0&0 \\
  0 & a_{11}& 0&0 \\
  0& 0&a_{11}&0 \\
  0&0&0&a_{11}
\end{array}%
\right)$&$1$&
$As_{4}^{12}$& $\left(%
\begin{array}{cccc}
  a_{11} & 0 & 0&0 \\
  0 & a_{11}& 0&0 \\
  0& 0&a_{11}&0 \\
  0&0&0&a_{11}
\end{array}%
\right)$&$1$\\
\hline
$As_{4}^{13}$& $\left(%
\begin{array}{cccc}
  a_{11} & 0 & 0&0 \\
  0 & a_{11}& 0&0 \\
  0& 0&a_{33}&0 \\
  0&0&0&a_{33}
\end{array}%
\right)$&$2$&
$As_{4}^{14}$& $\left(%
\begin{array}{cccc}
   a_{11} & 0 & 0&0 \\
  0 & a_{11}& 0&0 \\
  0& 0&a_{11}&0 \\
  0&0&0&a_{11}
\end{array}%
\right)$&$1$\\
\hline

$As_{4}^{15}$& $\left(%
\begin{array}{cccc}
   a_{11} & 0 & 0&0 \\
  0 & a_{11}& 0&0 \\
  0& 0&a_{11}&0 \\
  0&0&0&a_{11}
\end{array}%
\right)$&$1$&
$As_{4}^{16}$& $\left(%
\begin{array}{cccc}
   a_{11} & 0 & 0&0 \\
  0 & a_{11}& 0&0 \\
  0& 0&a_{11}&0 \\
  0&0&0&a_{11}
\end{array}%
\right)$ &$1$\\
\hline
$As_{4}^{17}$& $\left(%
\begin{array}{cccc}
  a_{11} & 0 & 0&0 \\
  0 & a_{11}& 0&0 \\
  0& 0&a_{11}&0 \\
  0&0&0&a_{11}
\end{array}%
\right)$&$1$&
$As_{4}^{18}$& $\left(%
\begin{array}{cccc}
  a_{11} & 0 & 0&0 \\
  0 & a_{11}& 0&0 \\
  0& 0&a_{33}&0 \\
  0&0&0&a_{33}
\end{array}%
\right)$&$2$\\
\hline
$As_{4}^{19}$& $\left(%
\begin{array}{cccc}
    a_{11} & 0 & 0&0 \\
  0 & a_{22}& 0&0 \\
  0& 0&a_{11}&0 \\
  0&0&0&a_{22}
\end{array}%
\right)$&$2$&
$As_{4}^{20}$& $\left(%
\begin{array}{cccc}
  a_{11} & 0 & 0&0 \\
  0 & a_{22}& 0&0 \\
  0& 0&a_{33}&0 \\
  0&0&0&a_{44}
\end{array}%
\right)$&$4$\\
\hline

$As_{4}^{21}$
& $\left(%
\begin{array}{cccc}
  a_{11}& 0 & 0&0 \\
  0 &a_{11}& 0&0 \\
 a_{31}& 0&a_{11}&0 \\
  a_{41}&a_{42}&a_{31}&a_{11}
\end{array}%
\right)$&$4$&
$As_{4}^{22}$& $\left(%
\begin{array}{cccc}
   a_{11}& 0 & 0&0 \\
  0 &a_{11}& 0&0 \\
 a_{31}& a_{32}&a_{11}&0 \\
  a_{41}&a_{42}&0&a_{11}
\end{array}%
\right)$&$5$\\
\hline
$As_{4}^{23}(\mu)$& $\left(%
\begin{array}{cccc}
 a_{11}& 0 & 0&0 \\
  0 &a_{11}& 0&0 \\
 a_{31}& a_{32}&a_{11}&0 \\
  a_{41}&a_{42}&0&a_{11}
\end{array}%
\right)$&$5$&
$As_{4}^{24}$& $\left(%
\begin{array}{cccc}
  a_{11}& 0 & 0&0 \\
  0 &a_{11}& 0&0 \\
 0& 0&a_{11}&0 \\
  a_{41}&a_{42}&a_{43}&a_{11}
\end{array}%
\right)$&$4$\\
\hline
\end{tabular}
\end{center}
\begin{flushright}
Continued to the next page
\end{flushright}
\end{table}
\clearpage
\begin{table}[ht!]
\begin{center}
\small
\begin{tabular}{|c|c|c||c|c|c|c|c|} \hline
\textbf{IC}&\textbf{Centroid }&\textbf{Dim}&\textbf{IC}&\textbf{Centroid }&\textbf{Dim}\\
& \textbf{$\Gamma(A)$} & && $\Gamma(A)$&  \\
\hline

$As_{4}^{25}$ & $\left(%
\begin{array}{cccc}
  a_{11}& 0 & 0&0 \\
  a_{21} &a_{22}& 0&a_{24} \\
 a_{31}& a_{32}&a_{11}&0 \\
  a_{41}&0&0&a_{11}
\end{array}%
\right)$&$7$&
$As_{4}^{26}$& $\left(%
\begin{array}{cccc}
 a_{11}& 0 & 0&0 \\
  a_{21} &a_{11}& 0&0 \\
 a_{31}& 0&a_{11}&0 \\
  a_{41}&0&0&a_{11}
\end{array}%
\right)$&$4$\\
\hline
$As_{4}^{27}$& $\left(%
\begin{array}{cccc}
 a_{11}& 0 & 0&0 \\
  0 &a_{11}& 0&0 \\
 0& 0&a_{11}&0 \\
  a_{41}&0&0&a_{11}
\end{array}%
\right)$&$2$&
$As_{4}^{28}$ & $\left(%
\begin{array}{cccc}
 a_{11}& 0 & 0&0 \\
  a_{21} &a_{11}& 0&0 \\
 0& 0&a_{11}&0 \\
  0&0&0&a_{11}
\end{array}%
\right)$&$2$\\
\hline
$As_{4}^{29}$ & $\left(%
\begin{array}{cccc}
a_{11}& 0 & 0&0 \\
0 &a_{11}& 0&0 \\
0& 0&a_{11}&0 \\
0&0&0&a_{11}
\end{array}%
\right)$&$1$&
$ As_{4}^{30}$ & $\left(%
\begin{array}{cccc}
  a_{11}& 0 & 0&0 \\
0 &a_{22}& 0&0 \\
0& 0&a_{11}&0 \\
0&0&0&a_{11}
\end{array}%
\right)$&$2$\\
\hline
$As_{4}^{31}$ & $\left(%
\begin{array}{cccc}
 a_{11}& 0 & 0&0 \\
0 &a_{11}& 0&0 \\
0& 0&a_{11}&0 \\
0&0&0&a_{11}
\end{array}%
\right)$&$1$&
$As_{4}^{32}$& $\left(%
\begin{array}{cccc}
  a_{11}& 0 & 0&0 \\
0 &a_{11}& 0&0 \\
0& 0&a_{11}&0 \\
0&0&0&a_{11}
\end{array}%
\right)$&$1$\\
\hline
$As_{4}^{33}$& $\left(%
\begin{array}{cccc}
  a_{11}& 0 & 0&0 \\
0 &a_{22}& 0&0 \\
0& 0&a_{22}&0 \\
0&0&0&a_{22}
\end{array}%
\right)$&$2$&
$As_{4}^{34}$& $\left(%
\begin{array}{cccc}
 a_{11}& 0 & 0&0 \\
0 &a_{22}& 0&0 \\
0& 0&a_{33}&0 \\
0&0&0&a_{33}
\end{array}%
\right)$&$1$\\
\hline
$As_{4}^{35}$& $\left(%
\begin{array}{cccc}
 a_{11}& 0 & 0&0 \\
0 &a_{11}& 0&0 \\
0& 0&a_{11}&0 \\
a_{41}&a_{42}&a_{43}&a_{11}
\end{array}%
\right)$&$4$&
$As_{4}^{36}$& $\left(%
\begin{array}{cccc}
a_{11}& 0 & 0&0 \\
0 &a_{11}& 0&0 \\
a_{31}& a_{32}&a_{11}&-a_{41} \\
a_{41}&a_{42}&a_{43}&a_{11}\\
\end{array}%
\right)$&$4$\\
\hline
$As_{4}^{37}$ & $\left(%
\begin{array}{cccc}
  a_{11}& 0 & 0&0 \\
0 &a_{11}& 0&0 \\
0& 0&a_{11}&0 \\
0&0&0&a_{11}
\end{array}%
\right)$&$1$&
$As_{4}^{38}$ & $\left(%
\begin{array}{cccc}
  a_{11}& 0 & 0&0 \\
a_{21} &a_{11}& 0&0 \\
a_{31}& 0&a_{11}&0 \\
0&0&0&a_{11}
\end{array}%
\right)$&$3$\\
\hline
$As_{4}^{39}$ & $\left(%
\begin{array}{cccc}
a_{11}& 0 & 0&0 \\
0 &a_{11}& 0&0 \\
a_{31}& 0&a_{11}&0 \\
0&a_{31}&0&a_{11}
\end{array}%
\right)$&$2$&
$As_{4}^{40}$ & $\left(%
\begin{array}{cccc}
  a_{11}& 0 & 0&0 \\
a_{21} &a_{11}& 0&a_{24} \\
a_{31}& 0&a_{11}&a_{34} \\
0&a_{31}&0&a_{11}
\end{array}%
\right)$&$5$\\
\hline
$As_{4}^{41}$ & $\left(%
\begin{array}{cccc}
  a_{11}& 0 & 0&0 \\
a_{21} &a_{11}& 0&a_{24} \\
0& 0&a_{33}&0 \\
0&0&a_{43}&a_{33}
\end{array}%
\right)$&$4$&
$As_{4}^{42}$  & $\left(%
\begin{array}{cccc}
a_{11}& 0 & 0&0 \\
0 &a_{11}& 0&0 \\
a_{31}& 0&a_{11}&0 \\
0&0&0&a_{11}
\end{array}%
\right)$&$2$\\
\hline

$As_{4}^{43}$ & $\left(%
\begin{array}{cccc}
  a_{11}& 0 & 0&0 \\
0 &a_{11}& 0&0 \\
a_{31}& 0&a_{11}&0 \\
0&a_{31}&0&a_{11}
\end{array}%
\right)$&$2$&
$As_{4}^{44}$& $\left(%
\begin{array}{cccc}
a_{11}& 0 & 0&0 \\
0 &a_{11}& 0&0 \\
a_{31}& 0&a_{11}&0 \\
0&0&0&a_{11}
\end{array}%
\right)$&$2$\\
\hline
\end{tabular}
\end{center}
\begin{flushright}
Continued to the next page
\end{flushright}
\end{table}
\clearpage
\begin{table}[ht!]
\begin{center}
\small
\begin{tabular}{|c|c|c||c|c|c|c|c|} \hline
\textbf{IC}&\textbf{Centroid }&\textbf{Dim}&\textbf{IC}&\textbf{Centroid }&\textbf{Dim}\\
& \textbf{$\Gamma(A)$} & && $\Gamma(A)$&  \\
\hline
$As_{4}^{45}$ & $\left(%
\begin{array}{cccc}
  a_{11}& 0 & 0&0 \\
0 &a_{11}& 0&0 \\
0& 0&a_{11}&0 \\
0&0&0&a_{11}
\end{array}%
\right)$&$1$&
$As_{4}^{46}$ & $\left(%
\begin{array}{cccc}
 a_{11}& 0 & 0&0 \\
0 &a_{11}& 0&0 \\
0& 0&a_{11}&0 \\
0&0&0&a_{11}
\end{array}%
\right)$&$1$\\
\hline
$ As_{4}^{47}$ & $\left(%
\begin{array}{cccc}
 a_{11}& 0 & 0&0 \\
0 &a_{22}& 0&0 \\
0& a_{32}&a_{22}&0 \\
0&a_{42}&0&a_{22}
\end{array}%
\right)$&$4$&
$As_{4}^{48}$ & $\left(%
\begin{array}{cccc}
 a_{11}& 0 & 0&0 \\
a_{21} &a_{11}& 0&0 \\
a_{31}&a_{21}&a_{11}&0 \\
0&a_{31}&a_{21}&a_{11}\\
\end{array}%
\right)$&$3$\\
\hline

$As_{4}^{49}$ & $\left(%
\begin{array}{cccc}
  a_{11}& 0& 0&0\\
  0&a_{11}&a_{23}&0 \\
 0& 0&a_{11}&0\\
  0&0&0&a_{11}
\end{array}%
\right)$&$3$&
$As_{4}^{50}$ & $\left(%
\begin{array}{cccc}
  a_{11}& 0& 0&0\\
  a_{21}&a_{11}&0&0 \\
a_{31}& 0&a_{11}&0\\
  a_{41}&0&0&a_{11}
\end{array}%
\right)$&$4$\\
\hline
$As_{4}^{51}$ & $\left(%
\begin{array}{cccc}
 a_{11}& 0& 0&0\\
  a_{21}&a_{11}&0&0 \\
 0& 0&a_{11}&0\\
  a_{41}&a_{21}&0&a_{11}
\end{array}%
\right)$&$3$&
$As_{4}^{52}$  & $\left(%
\begin{array}{cccc}
  a_{11}& 0& 0&0\\
  a_{21}&a_{11}&0&0 \\
 0& 0&a_{11}&0\\
  a_{41}&a_{21}&0&a_{11}
\end{array}%
\right)$&$3$\\
\hline
$As_{4}^{53}$  & $\left(%
\begin{array}{cccc}
 a_{11}& 0& 0&0\\
  0&a_{22}&0&0 \\
 0& a_{32}&a_{22}&0\\
  0&a_{42}&a_{32}&a_{22}
\end{array}%
\right)$&$4$&
$As_{4}^{54}$&$\left(%
\begin{array}{cccc}
  a_{11}& 0 & 0&0 \\
0 &a_{11}& 0&0 \\
0& 0&a_{11}&0 \\
0&0&0&a_{11}
\end{array}%
\right)$&$1$\\
\hline
$As_{4}^{55}$ & $\left(%
\begin{array}{cccc}
 a_{11}& 0 & 0&0 \\
 a_{21} &a_{11}& 0&0 \\
 a_{31}& 0&a_{11}&0 \\
 a_{41}& a_{21}&0&a_{11}
\end{array}%
\right)$&$4$&
$As_{4}^{56}$ & $\left(%
\begin{array}{cccc}
   a_{11}& 0 & 0&0 \\
 a_{21} &a_{11}& 0&0 \\
 a_{31}& 0&a_{11}&0 \\
 a_{41}& -a_{31}&-a_{21}&a_{11}
\end{array}%
\right)$&$4$\\
\hline
$As_{4}^{57}$& $\left(%
\begin{array}{cccc}
   a_{11}& 0 & 0&0 \\
0 &a_{11}& 0&0 \\
 0& 0&a_{11}&0 \\
 a_{41}& 0&0&a_{11}
\end{array}%
\right)$&$2$&
$As_{4}^{58}$& $\left(%
\begin{array}{cccc}
   a_{11}& 0 & 0&0 \\
 0 &a_{11}& 0&0 \\
 0& 0&a_{11}&0 \\
 a_{41}& 0&0&a_{11}
\end{array}%
\right)$&$2$\\
\hline
\end{tabular}
\end{center}
\end{table}

\begin{proof}
Let us consider $As_{4}^{1}$, the structure constants are given as follows $\gamma_{11}^{3}=1, \gamma_{22}^{4}=1$ others being zero. From the systems of equations (\ref{qaaq}) and (\ref{qaaaq}), we have\\
$a_{12}=a_{13} =a_{14}=a_{21}=a_{23}=a_{24}=a_{31}= a_{32}=a_{34}=a_{41}=a_{42}=a_{43}=0$ and $a_{33}=a_{11}=a_{44}.$
Therefore, we obtain the centroids of $As_{4}^{1}$ in matrix form as follows; $$\Gamma(As_{4}^{1})= \left(%
\begin{array}{cccc}
  a_{11} & 0 & 0&0 \\
  0 & a_{22}& 0&0 \\
  0&0&a_{11}&0 \\
  0&0&0&a_{11}
\end{array}%
\right).$$
Now let us consider $As_{4}^{2}$. The structure constants are given as follows $\gamma_{12}^{3}=1, \gamma_{21}^{4}=1$ . Using the system of equations (\ref{qaaq}) and (\ref{qaaaq}), we get
$a_{12}=a_{13} =a_{14}=a_{21}=a_{23}=a_{24}= a_{34}=a_{43}=0$ and $a_{33}=a_{11}=a_{22}=a_{44}.$\\\\
Therefore, $$\Gamma(As_{4}^{2})= \left(%
\begin{array}{cccc}
   a_{11} & 0 & 0&0 \\
   0 & a_{11}& 0&0 \\
  a_{31}&a_{32}&a_{11}&0 \\
   a_{41}&a_{42}& 0&a_{11}
\end{array}%
\right).$$
Let us take the class  $As_{4}^{3}$ with structure constants  $\gamma_{11}^{4}=1, \gamma_{31}^{4}=1$ and others are zero. Applying (\ref{qaaq}) and (\ref{qaaaq}), we get\\
$a_{12}=a_{13} =a_{14}=a_{21}=a_{23}=a_{24}= a_{31}= a_{32}= a_{34}=0$ and $a_{33}=a_{11}=a_{22}=a_{44}.$\\\\
Therefore, $$\Gamma(As_{4}^{3})= \left(%
\begin{array}{cccc}
  a_{11} & 0 & 0&0 \\
 0 & a_{11} & 0&0 \\
  0&0&a_{11} &0 \\
  a_{41}&a_{42}&a_{43}&a_{11}
\end{array}%
\right).$$
Using the same approach, the centroids of associative algebras in dimension four can be computed for the remaining algebras.
\end{proof}

\section{Acknowledgement}

Part of this study was presented at the regional fundamental science congress, 2014 (RFSC2014) organised by Universiti Putra Malaysia. The authors would like to thank Assoc. Prof., Dr. M. R. K. Ariffin for his financial support (which facilitated first author participation in RFSC2014) through the Putra Grant with vote number 9419300.


\begin{thebibliography}{99}

%
%
\bibitem{J} {N. Jacobson}, A note on automorphisms and derivations of Lie algebras, {\it Proc. Amer. Math. Soc.}, \textbf{6}(2), (1955), 281--283.

\bibitem{DL}{J. Dixmier, W.G. Lister}, Derivations of nilpotent Lie algebras. {\it Proc. Amer. Math. Soc.}, \textbf{8} (1957), 155--158.

\bibitem{RH3}  G. Leger, S. Togo, Characteristically nilpotent Lie algebras, {\it Duke Math. J.}, \textbf{26}(4), (1959), 623--628 .

\bibitem{T} {J. Tits}, Sur les constantes de structure et le theorem dexistence des algebres de Lie semisimples, {\it Publ. Math. I.H.E.S.}, \textbf{31}(1), (1966), 21--58. http://dx.doi.org/10.1007/BF02684801.

\bibitem{peirc} B. Pierce, Linear Associative algebras, {\it Amer. Math. Journal}, (1881) \textbf{4}(1) 97--229. http://www.jstor.org/stable/2369153 

\bibitem{hazlett} O.C. Hazlett, On the classification and invariantive characterization of nilpotent algebras,{\it American Journal of Mathematics}, {\bf 38}(2), (1916), 109--138. http://www.jstor.org/stable/2370262 

\bibitem{mazzola1979} G. Mazzola, The algebraic and geometric classification of associative algebras of dimension five, {\it Manuscripta Math}., {\bf 27}(1), (1979), 81-–101. http://dx.doi.org/10.1007/BF01297739.

\bibitem{mazzola1980} G. Mazzola, Generic finite schemes and Hochschild cocycles, {\it Commentarii Mathematici Helvetici}., {\bf 55}(1), (1980), 267--293 . http://dx.doi.org/10.1007/BF02566686.

\bibitem{poonen2008} B. Poonen, Isomorphism types of commutative algebras of finite rank, \emph{Computational Arithmetic Geometry}, {\bf  463}(2008), 111--120.

\bibitem{RRB}{I.S. Rakhimov, I.M. Rikhsiboev, W. Basri,} Complete lists of low dimensional complex associative algebras, (2009). {\it arXiv:0910.0932v1 [math.RA].}

\bibitem{BN} G. Benkart, E. Neher, The centroid of extended affine and root graded Lie algebras, {\it Journal of Pure and Applied Algebra}, {\bf 205}(1), (2006), 117--145.

\bibitem{Mel} D.J. Melville, Centroids of nilpotent Lie algebras, {\it Comm. Algebra}, \textbf{20}(12), (1992), 3649--3682.

\bibitem{FRS} M.A. Fiidow, I.S. Rakhimov, S.K. Said Hussain, Derivations and Centroids of Associative Algebras. {\it IEEE Proceedings of International Conference on Research and Education in Mathematics (ICREM7)}. (2015), 227--232. http://dx.doi.org/10.1109/ICREM.2015.7357059.
\end{thebibliography}
\end{document}